\documentclass[final]{dmtcs-episciences}

\usepackage[utf8]{inputenc}
\usepackage{subfigure}

\usepackage[round]{natbib}
\author{Michael Albert\affiliationmark{1}
  \and Robert Brignall\affiliationmark{2}}
\title[2 $\times$ 2 monotone grid classes are finitely based]{2 $\times$ 2 monotone grid classes are finitely based}
\affiliation{%
  Department of Computer Science, University of Otago, New Zealand\\
  Department of Mathematics and Statistics, The Open University, England}
\received{2015-11-3}
\accepted{2016-1-7}

\usepackage{amsmath}
\usepackage{amsthm}
\usepackage{amssymb}
\usepackage{latexsym}
\usepackage{textcomp}
\usepackage{enumerate}
\usepackage{color}
\usepackage{tikz}
\usetikzlibrary{matrix, calc, arrows}

\definecolor{lightgray}{rgb}{0.9, 0.9, 0.9}
\definecolor{darkgray}{rgb}{0.7, 0.7, 0.7}
\definecolor{darkblue}{rgb}{0, 0, .4}

\hypersetup{
    colorlinks=true,
    linkcolor=darkblue,
    anchorcolor=darkblue,
    citecolor=darkblue,
    urlcolor=darkblue,
    pdfpagemode=UseThumbs,
    pdftitle={2x2 monotone grid classes are finitely based},
    pdfsubject={Combinatorics},
    pdfauthor={M. Albert and R. Brignall},
    pdfkeywords={grid class, basis, permutation, pattern}
}
\theoremstyle{plain}
\newtheorem{theorem}{Theorem}[section]

\newtheorem{lemma}[theorem]{Lemma}
\newtheorem{observation}[theorem]{Observation}

\theoremstyle{definition}

\theoremstyle{remark}

\newcommand{\Av}{\operatorname{Av}}

\newcommand{\A}{\mathcal{A}}
\newcommand{\B}{\mathcal{B}}
\newcommand{\C}{\mathcal{C}}
\newcommand{\D}{\mathcal{D}}
\newcommand{\E}{\mathcal{E}}
\newcommand{\F}{\mathcal{F}}

\newcommand{\M}{\mathcal{M}}

\newcommand{\hjuxta}[2]{\left[\begin{array}{cc}#1&#2\end{array}\right]}

\newcommand{\Grid}{\operatorname{Grid}}

\tikzstyle{every node}=[circle, draw, fill=black,
                        inner sep=0pt, minimum width=4pt]
\tikzset{invis_nd/.style={fill=none,draw=none}}
\def\tgrid(#1,#2){
\draw [thin, gray] (0,0) grid (#1,#2);
}
\def\tup(#1,#2){
\pgfmathsetmacro\xa{#1 - 0.85};
\pgfmathsetmacro\ya{#2 - 0.85};
\pgfmathsetmacro\xb{#1 - 0.15};
\pgfmathsetmacro\yb{#2 - 0.15};
\draw [thick] (\xa,\ya) -- (\xb,\yb);
}
\def\tdown(#1,#2){
\pgfmathsetmacro\xa{#1 - 0.85};
\pgfmathsetmacro\ya{#2 - 0.15};
\pgfmathsetmacro\xb{#1 - 0.15};
\pgfmathsetmacro\yb{#2 - 0.85};
\draw [thick] (\xa,\ya) -- (\xb,\yb);
}
\def\tclass(#1,#2)#3{
\pgfmathsetmacro\x{#1 - 0.5};
\pgfmathsetmacro\y{#2 - 0.5};
\node[draw=none, fill=none] at (\x,\y) {#3};
}
\def\twobytwo{\tikz[scale=0.2,baseline=0.2cm-0.5ex]{\tgrid(2,2);}}

\DeclareRobustCommand\squintA{
\tikz[scale=0.2,baseline=0.2cm-0.5ex]{
\draw [thin, gray] (0,0) rectangle (2,2);
\draw [thin, gray] (1,0) -- (1,2);
\draw [thin, gray] (0,0.8) -- (1,0.8);
\draw [thin, gray] (1,1.2) -- (2,1.2);
}
}
\DeclareRobustCommand\squintB{
\tikz[scale=0.2,baseline=0.2cm-0.5ex]{
\draw [thin, gray] (0,0) rectangle (2,2);
\draw [thin, gray] (1,0) -- (1,2);
\draw [thin, gray] (0,1.2) -- (1,1.2);
\draw [thin, gray] (1,0.8) -- (2,0.8);
}
}

\begin{document}
\publicationdetails{18}{2016}{2}{1}{1325}
\maketitle

\begin{abstract}In this note, we prove that all $2 \times 2$ monotone grid classes are finitely based, i.e., defined by a finite collection of minimal forbidden permutations. This follows from a slightly more general result about certain $2 \times 2$ (generalised) grid classes that have two monotone cells in the same row. 
\end{abstract}

\section{Introduction}

In recent years, the emerging theory of grid classes has led to some of the major structural and enumerative developments in the study of permutation patterns. Particular highlights include the characterisation of all possible ``small'' growth rates~\citep{huczynska:grid-classes-an:,kaiser:on-growth-rates:,vatter:small-permutati:} and the subsequent result that all classes with these growth rates have rational generating functions~\citep{albert:inflations-of-g:}. 

To support results such as these, the study of grid classes themselves has gained importance. Restricting one's attention to \emph{monotone} grid classes, it is known that the structure of the matrix defining a grid class determines both its growth rate~\citep{bevan:growth-rates:}, and whether it is well-partially-ordered~\citep{murphy:profile-classes:}. 

One remaining open question about monotone grid classes concerns their \emph{bases}, that is, the sets of minimal forbidden permutations of the classes. Backed up by some computational evidence, it is widely believed that all monotone grid classes are finitely based, but this is only known to be true for certain families, most notably those whose row-column graphs\footnote{The row-column graph of a $\{0,\pm1\}$-matrix $\M$ is the bipartite graph whose biadjacency matrix has $ij$-th entry equal to $|\M_{ij}|$.} are \emph{forests}~\citep{albert:geometric-grid-:}. To date, the only other instances of monotone grid classes that are known to have a finite basis are two $2\times 2$ grid classes. The first concerns the class of \emph{skew-merged} permutations, $\Av(2143,3412)$, in~\citep{stankova:forbidden-subse:}, while the second is in Waton's PhD thesis~\citep{waton:on-permutation-cl:}. Inspired by Waton's approach, we show that a certain family of (non-monotone) $2\times 2$ grid classes are all finitely based, from which we can conclude the following result.

\begin{theorem}\label{thm-twobytwo-mono}
Every $2\times 2$ monotone grid class is finitely based.
\end{theorem}

The rest of this section covers a number of prerequisite definitions. In Section~\ref{sec-relative-bases} we introduce a more general construction than grid classes, based on juxtapositions, that are known to be finitely based, and use these to characterise the grid classes they contain. In Section~\ref{sec-main} we consider three separate cases that will enable us to prove our more general result (Theorem~\ref{thm-twobytwo}), and thence Theorem~\ref{thm-twobytwo-mono}.

Writing permutations in one-line notation, we say that the permutation $\sigma$ is \emph{contained} in a permutation $\pi$, denoted $\sigma\leq\pi$, if there is a subsequence of the entries of $\pi$ that have the same relative ordering as the entries of $\sigma$. A specific instance of a set of entries of $\pi$ witnessing this containment is called a \emph{copy} of $\sigma$ in $\pi$. Containment forms a partial order on the set of all permutations, and sets of permutations which are closed downwards in this order are called \emph{permutation classes}. Specifically, if $\C$ is a permutation class, $\pi\in\C$ and $\sigma\leq\pi$, then we must have $\sigma\in\C$. For convenience later, we regard the empty permutation as belonging to every permutation class.

While permutation classes can be defined in a number of ways (for example, the set of all permutations that can be sorted by a stack forms a permutation class), a convenient characterisation can be given in terms of the unique set of minimal forbidden permutations that do \emph{not} lie in the class. We call the set $B$ the \emph{basis} of a class $\C$ if
\[
	\C = \{\pi: \beta\not\leq\pi\text{ for all }\beta\in B\},
\]
and $B$ is minimal with this property, and we write $\C=\Av(B)$. By its minimality, the set $B$ must form an antichain under $\leq$, but since infinite antichains are know to exist in the containment partial order, $B$ need not be finite. When the basis of $\C$ is finite, we say that $\C$ is \emph{finitely based}.

We frequently make use of a graphical perspective, in which we represent a permutation $\pi$ by plotting the points $(i,\pi(i))$ ($i=1,\dots,|\pi|$) in the plane. Indeed, we do not distinguish between the permutation $\pi$ written in one-line notation, and the graphical representation of $\pi$.

For $m,n\geq 1$, let $\M$ be an $m\times n$ matrix whose entries are permutation classes (including possibly the empty class). The \emph{grid class} of the matrix $\M$, denoted $\Grid(\M)$, is the permutation class consisting of all permutations $\pi$ for which (in the graphical perspective) there exist $m-1$ horizontal and $n-1$ vertical lines which divide the entries of $\pi$ into $mn$ rectangles, so that the (possibly zero) entries of $\pi$ in each rectangle form a copy of a permutation from the class in the corresponding entry of $\M$. When the entries of $\M$ are all either $\Av(12)$, $\Av(21)$ or $\emptyset$, then $\Grid(\M)$ is a \emph{monotone} grid class.

We are mostly concerned with $2\times 2$ matrices in this paper, and in this case it will prove convenient to refer to these grid classes more succinctly. If $\M = \begin{pmatrix}\A&\B\\\C&\D\end{pmatrix}$ is a matrix consisting of permutation classes, then we write 
\[
	\tikz[scale=0.5,baseline=0.5cm-0.5ex]{\tgrid(2,2);\tclass(1,2){$\A$};\tclass(2,2){$\B$};\tclass(1,1){$\C$};\tclass(2,1){$\D$};}
	\]
to mean $\Grid(\M)$. Additionally, when (say) $\A = \Av(21)$, then we may refer to the cell \tikz[scale=0.5,baseline=0.2cm-0.5ex]{\tgrid(1,1);\tclass(1,1){$\A$};} using \tikz[scale=0.5,baseline=0.2cm-0.5ex]{\tgrid(1,1);\tup(1,1);}, reflecting the fact that all points in this cell are increasing. Similarly, we may write \tikz[scale=0.5,baseline=0.2cm-0.5ex]{\tgrid(1,1);\tdown(1,1);} when $\A=\Av(12)$. Finally, where the entries of the $2\times 2$ matrix $\M$ are either arbitrary or clear from the context, we may also simply refer to $\Grid(\M)$ as\ \twobytwo. 

We are ready to state our general theorem, from which Theorem~\ref{thm-twobytwo-mono} will follow.

\begin{theorem}\label{thm-twobytwo}
Let $\C$ and $\D$ be finitely based permutation classes. Then the three grid classes
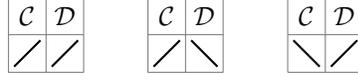
\begin{center}
\begin{tabular}{ccccc}
\tikz[scale=0.5]{
\tgrid(2,2);
\tup(1,1);
\tup(2,1);
\tclass(1,2){$\C$};
\tclass(2,2){$\D$};
}
&\hspace{20pt}&
\tikz[scale=0.5]{
\tgrid(2,2);
\tup(1,1);
\tdown(2,1);
\tclass(1,2){$\C$};
\tclass(2,2){$\D$};
}
&\hspace{20pt}&
\tikz[scale=0.5]{
\tgrid(2,2);
\tdown(1,1);
\tup(2,1);
\tclass(1,2){$\C$};
\tclass(2,2){$\D$};
}
\end{tabular}
\end{center}
are all finitely based. 
\end{theorem}

Our approach makes use of an existing result, which although not originally presented in this way, can be cast in terms of grid classes. For permutation classes $\C$ and $\D$, the \emph{(horizontal) juxtaposition} of $\C$ and $\D$ is the $1\times 2$ grid class\ \tikz[scale=0.5,baseline=0.2cm-0.5ex]{\tgrid(2,1);\tclass(1,1){$\C$};\tclass(2,1){$\D$};}. Similarly, the \emph{vertical juxtaposition} of $\C$ and $\D$ is the $2\times 1$ grid class \tikz[scale=0.5,baseline=0.5cm-0.5ex]{\tgrid(1,2);\tclass(1,2){$\C$};\tclass(1,1){$\D$};}.

\begin{lemma}[\citealp{atkinson:restricted-perm:}]\label{lem-juxt-basis}
Whenever $\C$ and $\D$ are finitely based, so are the horizontal and vertical juxtapositions of $\C$ and $\D$.
\end{lemma}

For clarity, we occasionally write $\hjuxta \C\D$ for the horizontal juxtaposition\ \tikz[scale=0.5,baseline=0.2cm-0.5ex]{\tgrid(2,1);\tclass(1,1){$\C$};\tclass(2,1){$\D$};} (we do not need the corresponding vertical juxtaposition notation).

%
%
%
%
%
%
%
%
%
%
%

\section{Juxtapositions and relative bases}\label{sec-relative-bases}

In this section, we give a characterisation of $2\times 2$ grid classes of the form
\[\E = \tikz[scale=0.5,baseline=0.5cm-0.5ex]{\tgrid(2,2);\tclass(1,2){$\A$};\tclass(2,2){$\B$};\tclass(1,1){$\C$};\tclass(2,1){$\D$};}\]
where $\A$, $\B$, $\C$ and $\D$ are four fixed (but arbitrary) permutation classes. 

We begin by considering the following related class, formed by the horizontal juxtaposition of two vertical juxtapositions:
\[\F = \hjuxta{\tikz[scale=0.5,baseline=0.5cm-0.5ex]{\tgrid(1,2);\tclass(1,2){$\A$};\tclass(1,1){$\C$};}}{\tikz[scale=0.5,baseline=0.5cm-0.5ex]{\tgrid(1,2);\tclass(1,2){$\B$};\tclass(1,1){$\D$};}}.\]
Note that if $\A,\B,\C$ and $\D$ are finitely based, then by repeated application of Lemma~\ref{lem-juxt-basis} so too is $\F$. 

Clearly, $\E\subseteq \F$. We are interested in the basis of $\E$, which we can separate into two parts: those basis elements of $\E$ that lie within $\F$, and those basis elements of $\E$ that are not in $\F$. By minimality and since $\E\subseteq\F$, this latter set must also be basis elements of $\F$. The set of basis elements of $\E$ that are contained in $\F$ we call the \emph{relative basis} of $\E$ in $\F$, and we have the following observation.  

\begin{observation}\label{obs-rel-basis}
Let $\C$ and $\D$ be two permutation classes such that $\D$ finitely based, and $\C\subseteq\D$. Then $\C$ is finitely based if and only if the relative basis of $\C$ in $\D$ is finite. 
\end{observation}

Consider any permutation $\pi$ in the set $\F\setminus \E$. Since $\pi$ lies in the juxtaposition class $\F$, we can write $\pi=\pi_1\pi_2$ with \[\pi_1\in 
\tikz[scale=0.5,baseline=0.5cm-0.5ex]{\tgrid(1,2);\tclass(1,2){$\A$};\tclass(1,1){$\C$};} \text{ and }\pi_2\in 
\tikz[scale=0.5,baseline=0.5cm-0.5ex]{\tgrid(1,2);\tclass(1,2){$\B$};\tclass(1,1){$\D$};}.\] 
We refer to the division line $v$ that separates $\pi_1$ from $\pi_2$ as a \emph{v-line}. Additionally, any horizontal division line in $\pi_1$ that demonstrates $\pi_1$ as a member of the vertical juxtaposition is called a \emph{left h-line} of $\pi$, and similarly any valid horizontal division line in $\pi_2$ is called a \emph{right h-line}. Thus, we can recognise $\pi\in\F$ by means of a \emph{division triple}, $(v,r,\ell)$, where $v$ is the v-line, $r$ the right h-line, and $\ell$ the left h-line.

The condition that $\pi\in\F\setminus\E$ can now be described as follows: for every division triple $(v,r,\ell)$ that recognises $\pi\in\F$, the right h-line $r$ and the left h-line $\ell$ cannot be at the same height. We use the symbol \squintA to denote the set of permutations in $\F$ which have a division triple $(v,r,\ell)$ where $\ell$ is no higher than $r$, and \squintB to denote those permutations which have a division where $\ell$ is no lower than $r$. Note that \squintA and \squintB are both in fact permutation classes, and also that $\F = \squintA \cup \squintB$.

Our main result of this section now follows. It shows in particular that $\pi\in\F\setminus\E$ cannot simultaneously lie in \squintA and \squintB, and hence the relative basis of $\E$ in $\F$ can be divided into two disjoint  parts: those that lie in \squintA and those that lie in \squintB.

\begin{lemma}\label{lem-squint-intersection}
Any $2\times 2$ grid class $\E=\twobytwo$ is equal to the intersection of the corresponding classes $\squintA$ and $\squintB$. That is, \[\E = \twobytwo = \squintA \cap \squintB\,.\]
\end{lemma}

\begin{proof}
First, it is clear that $\twobytwo\subseteq \squintA \cap \squintB$, so suppose that we have a permutation $\pi$ in $\squintA \cap \squintB$.

\begin{figure}
\centering
\begin{tikzpicture}[scale=0.4]
\draw (0,0) rectangle (10,10);
\draw [dashed] (4,0) -- (4,10);
\draw [dashed] (0,7) -- (4,7);
\draw [dashed] (4,6) -- (10,6);
\draw (6,0) -- (6,10);
\draw (0,3) -- (6,3);
\draw (6,5) -- (10,5);
\draw [gray, ->] (.5,7.5) -- (.5,7);
\draw [gray, ->] (9.5,5.5) -- (9.5,5);
\node[fill=none,draw=none] (r) at (10,5) [label=right:$r$] {};
\node[fill=none,draw=none] (r) at (10,6) [label=right:$r'$] {};
\node[fill=none,draw=none] (r) at (0,3) [label=left:$\ell$] {};
\node[fill=none,draw=none] (r) at (0,7) [label=left:$\ell'$] {};
\node[fill=none,draw=none] (r) at (6,10) [label=above:$v$] {};
\node[fill=none,draw=none] (r) at (4,10) [label=above:$v'$] {};

\end{tikzpicture}
\caption{The relationship between the division $(v,r,\ell)$ and $(v',r',\ell')$ in the proof of Lemma~\ref{lem-squint-intersection}. The small arrows indicate that the corresponding division lines have been chosen to be extremal in the direction specified by the arrows.}
\label{fig-squintgrids}
\end{figure}
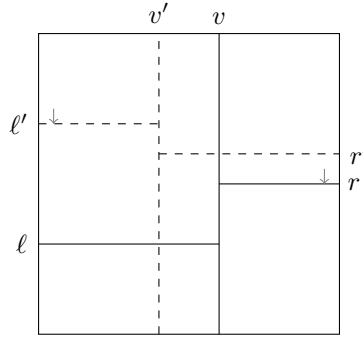

Consider $\pi$ first as a member of \squintA. There exists at least one division triple $(v,r,\ell)$ which recognises this, and we choose any valid v-line $v$, together with the lowest right h-line $r$ and the highest left h-line $\ell$. Note in particular that for any right h-line that is lower than $r$, there must exist a basis element in the top right cell. If $\ell$ and $r$ coincide, then we have $\pi\in \twobytwo$ and we are done, so we may assume that $\ell$ is strictly lower than $r$.

Next, consider $\pi$ as an element of \squintB. We pick a division $(v',r',\ell')$ by first choosing any v-line $v'$ which either coincides with $v$ or lies further to the left (the case where $v'$ is to the right of $v$ will follow upon rotating the picture by $180^\circ$). Next choose any valid $r'$, 
noting that $r'$ must be at least as high as $r$ to avoid introducing a basis element into the top right cell. Finally, choose $\ell'$ to be as low as possible, subject to the division triple $(v',r',\ell')$ remaining a valid division for membership of \squintB (see Figure~\ref{fig-squintgrids}). We claim that $\ell'$ is at the same height as $r'$.

Suppose, for a contradiction, that $\ell'$ lies strictly above $r'$, and let $\ell''$ be the left h-line that has the same height as $r'$. Since the division triple $(v',r',\ell'')$ does not witness $\pi\in \squintB$ (but $(v',r',\ell')$ does), there must exist some basis element in the top left region defined by $(v',r',\ell'')$. However, this region is contained in the top left region defined by $(v,r,\ell)$, so this is impossible. 

Thus $\ell'$ has the same height as $r'$, and $(v',r',\ell')$ is a division triple that recognises $\pi\in\squintB$, and hence $\pi\in\twobytwo$.
\end{proof}

\section{Main results}\label{sec-main}

We are ready to start proving our three main results.

\begin{lemma}\label{lem-casea}
For finitely based classes $\C$ and $\D$, the class
\[\E=
\tikz[scale=0.5,baseline=0.5cm-0.5ex]{
\tgrid(2,2);
\tup(1,1);
\tup(2,1);
\tclass(1,2){$\C$};
\tclass(2,2){$\D$};
}\]
is finitely based.\end{lemma}

\begin{proof}
First, let $B$ denote the relative basis of $\E$ inside the juxtaposition
\[\F=
\hjuxta{\tikz[scale=0.5,baseline=0.5cm-0.5ex]{
\tgrid(1,2);
\tup(1,1);
\tclass(1,2){$\C$};
}}{\tikz[scale=0.5,baseline=0.5cm-0.5ex]{
\tgrid(1,2);
\tup(1,1);
\tclass(1,2){$\D$};
}}.\]
Since $\F$ is finitely based, by Observation~\ref{obs-rel-basis} it suffices to show that $B$ is finite. By Lemma~\ref{lem-squint-intersection} and the comments preceding it, any $\pi\in B$ lies in exactly one of \squintA or \squintB. Consider first the case where $\pi\in \squintA$. We will identify a bounded number of points in $\pi$ that demonstrate $\pi\not\in\E$.

We begin by identifying two division triples, $(v_L,r_L,\ell_L)$ and $(v_R,r_R,\ell_R)$: $v_L$ is the leftmost v-line recognising $\pi\in\squintA$, and $v_R$ is the rightmost such v-line. Subject to these choices, we pick $\ell_L$ and $\ell_R$ to be as high as possible, and $r_L$ and $r_R$ as low as possible.

We now prove the following claim: if $(v,r,\ell)$ is any other division triple recognising $\pi\in\squintA$ where the left h-line $\ell$ is chosen as high as possible, then $\ell$ is at the same height as either $\ell_L$ or $\ell_R$. 

If $\ell_L$ and $\ell_R$ are at the same height, the claim follows immediately, so we can assume that $\ell_L$ is strictly higher than $\ell_R$. The situation is as depicted in Figure~\ref{fig-highest-left-lines}: we identify four points, $a$, $b$, $c$ and $d$, which are distinct (except possibly $b=c$) and which form the copies of 21 that define $\ell_L$ and $\ell_R$. Note that $a$ and $c$ lie immediately above $\ell_L$ and $\ell_R$, and, except that the relative positions of $a$ and $c$ can be interchanged providing $b\neq c$, the points must be arranged in the way shown in Figure~\ref{fig-highest-left-lines} in order that $\pi\in\squintA$. For the same reason, all other points of $\pi$ that lie in the marked rectangular regions 1, 2, 3 and 4 (defined by the bounding dotted and dashed lines) in Figure~\ref{fig-highest-left-lines} must lie on the diagonal segments indicated.

\begin{figure}
\centering
\begin{tikzpicture}[scale=0.4]
\draw [gray, dashed] (3,0) -- (3,10);
\draw [gray, dashed] (0,7) -- (3,7);
\draw [gray, dashed] (3,9) -- (10,9);
\draw [gray, dotted] (7,0) -- (7,10);
\draw [gray, dotted] (0,3) -- (7,3);
\draw [gray, dotted] (7,5) -- (10,5);
\draw [gray, ->] (3.5,2) -- (3,2);
\draw [gray, ->] (6.5,2) -- (7,2);
\draw [gray, ->] (.2,6.5) -- (.2,7);
\draw [gray, ->] (.2,2.5) -- (.2,3);
\draw [gray, ->] (9.8,9.5) -- (9.8,9);
\draw [gray, ->] (9.8,5.5) -- (9.8,5);
\node at (.5,7.3) [label=45:{\small $a$}] {};
\node (b) at (2.5, 6.3) [label=left:{\small $b$}] {};
\node (c) at (1,3.3) [label=left:{\small $c$}] {};
\node (d) at (5,2.5) [label=right:{\small $d$}] {};
\draw (b) -- (c);
\draw (d) -- (4.45,1.4);
\draw (0,0) -- (.7,1.4);
\draw (5.3,3.5) -- (6.5,5.9);
\node at (10,9) [invis_nd,label=right:$r_L$] {};
\node at (10,5) [invis_nd,label=right:$r_R$] {};
\node at (0,7) [invis_nd,label=left:$\ell_L$] {};
\node at (0,3) [invis_nd,label=left:$\ell_R$] {};
\node at (3,10) [invis_nd,label=above:$v_L$] {};
\node at (7,10) [invis_nd,label=above:$v_R$] {};
\node [fill=none,inner sep=1pt] at (6,8) {\tiny $1$};
\node [fill=none,inner sep=1pt] at (0,5) {\tiny $2$};
\node [fill=none,inner sep=1pt] at (2,0) {\tiny $3$};
\node [fill=none,inner sep=1pt] at (6,0) {\tiny $4$};
\end{tikzpicture}
\caption{The relationships between the division triples $(v_L,r_L,\ell_L)$ and $(v_R,r_R,\ell_R)$, the points defining $\ell_L$ and $\ell_R$, and the restrictions on the placement of points in the four rectangular regions 1---4.}
\label{fig-highest-left-lines}
\end{figure}
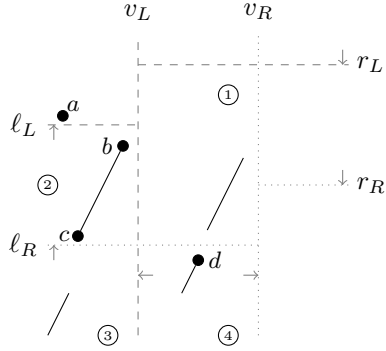

Consider any division triple $(v,r,\ell)$ recognising $\pi\in\squintA$ where $\ell$ is chosen as high as possible. If $v$ lies further left than all points in the region labelled 4 in Figure~\ref{fig-highest-left-lines}, then we can choose $\ell$ at the same height as $\ell_L$. On the other hand, if any point from region 4 lies to the left of $v$, then $c$ must lie above $\ell$, and thus $\ell$ is at the same height as $\ell_R$. This completes the claim.

We can now identify the following bounded collection of points of $\pi$: (i) a basis element of \tikz[scale=0.4,baseline=0.4cm-0.5ex]{
\tgrid(1,2);
\tup(1,1);
\tclass(1,2){\footnotesize$\C$};}
which defines $v_R$, (ii) a basis element of
\tikz[scale=0.4,baseline=0.4cm-0.5ex]{
\tgrid(1,2);
\tup(1,1);
\tclass(1,2){\footnotesize$\D$};}
to define $v_L$, and (iii) at most 4 points $a,b,c$ and $d$ defining the two left h-lines $\ell_R$ and $\ell_L$. 

It remains to identify a bounded number of points to ensure that any division triple $(v,r,\ell)$ recognising $\pi\in\squintA$ has $\ell$ strictly lower than $r$. For this, it suffices to consider only the \emph{extremal} triples $(v,r,\ell)$ where $\ell$ is as high as possible, and $r$ is as low as possible. We identify the extremal triple $(v_X,r_X,\ell_X)$ where the v-line $v_X$ is chosen to lie immediately to the left of all points in region 4 of Figure~\ref{fig-highest-left-lines}. By the earlier claim, $\ell_X$ has the same height as $\ell_L$. The lowest right h-line $r_X$ must lie strictly above $\ell_X$, and is defined by a basis element of $\D$ to the right of $v_X$, with one point lying immediately below $r_X$. Observe that for any extremal triple $(v,r,\ell)$ where $v$ lies to the left of $v_X$, we have that $\ell$ is at the same height as $\ell_X$, and $r$ can be no lower than $r_X$. In particular, since $\pi$ as a basis element is minimally not in $\E$, if $r$ is higher than $r_X$ then it is because of points in $\pi$ that we have already identified.

Similarly, the position of the line $r_R$ is fixed by a basis element of $\D$ to the right of $v_R$. For any extremal triple $(v,r,\ell)$ where $v$ is further right than $v_X$, we know that $\ell$ is at the same height as $\ell_R$, and $r$ can be no lower than $r_R$ (because of the basis element of $\D$). Thus, again by the minimality of $\pi$, if $r$ is strictly higher than $r_R$ it is because of points that we have already identified.

From this, we conclude that if $\pi\in\squintA$ is a basis element of $\E$ relative to $\F$ then the number of points in $\pi$ is bounded, as $\pi$ comprises the points identified in (i), (ii) and (iii) above, and by at most two basis elements of $\D$. 

\begin{figure}
\centering
\begin{tabular}{ccc}
\begin{tikzpicture}[scale=0.4]
\draw [gray, dashed] (3,0) -- (3,10);
\draw [gray, dashed] (0,3) -- (3,3);
\draw [gray, dashed] (3,3) -- (10,3);
\draw [gray, dotted] (7,0) -- (7,10);
\draw [gray, dotted] (0,7) -- (7,7);
\draw [gray, dotted] (7,5) -- (10,5);
\draw [gray, ->] (3.5,9) -- (3,9);
\draw [gray, ->] (6.5,9) -- (7,9);
\draw [gray, ->] (.2,7.5) -- (.2,7);
\draw [gray, ->] (7.2,2.5) -- (7.2,3);
\draw [gray, ->] (7.2,4.5) -- (7.2,5);
\node (a) at (4.3,3.3) {};
\node (b) at (5.5,6.7) {};
\node (c) at (9.5,4.7) {};
\node (d) at (8,5.3) {};
\node (e) at (8.3,2.7) {};
\draw (a) -- (b);
\draw (c) -- (8.7,3.3);
\draw (4,1.7) -- (3.3,1);
\draw (0,0) -- (2.7,.7);
\draw (e) -- (7.3,2);
\node at (10,3) [invis_nd,label=right:$r_L$] {};
\node at (10,5) [invis_nd,label=right:$r_R$] {};
\node at (0,7) [invis_nd,label=left:{$\ell_R$}] {};
\node at (0,3) [invis_nd,label=left:{$\ell_L$}] {};
\node at (3,10) [invis_nd,label=above:$v_L$] {};
\node at (7,10) [invis_nd,label=above:$v_R$] {};  
\node [invis_nd] at (1.5,5) {$\varnothing$};
\end{tikzpicture}
&\rule{10pt}{0pt}&
\begin{tikzpicture}[scale=0.4]
\draw [gray, dashed] (3,0) -- (3,10);
\draw [gray, dashed] (0,5) -- (3,5);
\draw [gray, dashed] (3,4) -- (10,4);
\draw [gray, dotted] (7,0) -- (7,10);
\draw [gray, dotted] (0,7) -- (7,7);
\draw [gray, dotted] (7,4) -- (10,4);
\draw [gray, ->] (3.5,2) -- (3,2);
\draw [gray, ->] (6.5,2) -- (7,2);
\draw [gray, ->] (.2,7.5) -- (.2,7);
\draw [gray, ->] (.2,5.5) -- (.2,5);
\draw [gray, ->] (7.2,3.5) -- (7.2,4);
\node (a) at (2.5,4.7) {};
\node (b) at (5.5,6.7) {};
\node (c) at (9.5,3.7) {};
\node (d) at (8,4.3) {};
\draw (a) -- (0,0);
\draw (b) -- (3.5,5.3);
\draw (c) -- (7.3,0);
\node at (10,4) [invis_nd,label=right:{$r_L=r_R$}] {};
\node at (0,5) [invis_nd,label=left:{$\ell_L$}] {};
\node at (0,7) [invis_nd,label=left:{$\ell_R$}] {};
\node at (3,10) [invis_nd,label=above:$v_L$] {};
\node at (7,10) [invis_nd,label=above:$v_R$] {};
\node [invis_nd] at (5,1.5) {$\varnothing$};
\node [invis_nd] at (1.5,6) {$\varnothing$};
\end{tikzpicture}
\end{tabular}
\caption{The relationships between the division triples $(v_L,r_L,\ell_L)$ and $(v_R,r_R,\ell_R)$ when $\pi\in\squintB$. On the left, if $r_L$ and $r_R$ are at different heights, then $\ell_L$ is at the same height as $r_L$. On the right, if $r_L$ and $r_R$ are at the same height, then the points defining $\ell_L$ guarantee $\pi\not\in\E$ for every triple $(v,r,\ell)$ recognising $\pi\in\squintB$.}
\label{fig-squintb}
\end{figure}
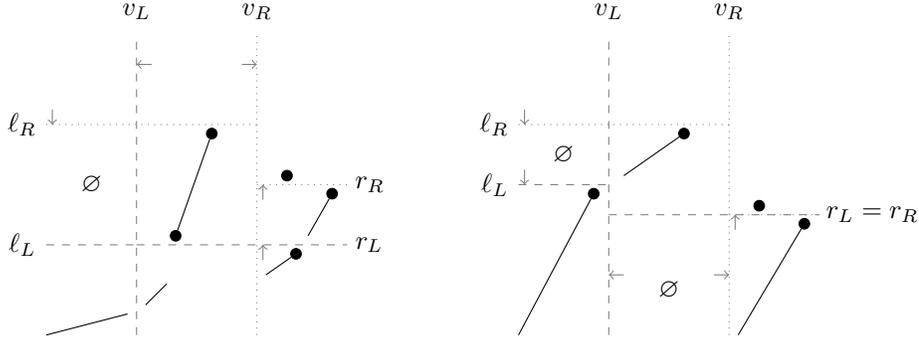
The argument for a basis element $\pi$ that lies in \squintB\ is similar, and we omit some of the details. The process begins by identifying the leftmost and rightmost v-lines $v_L$ and $v_R$, and the corresponding highest right h-lines $r_L$ and $r_R$. The left hand picture in Figure~\ref{fig-squintb} illustrates that $r_L$ and $r_R$ cannot have different heights (else $\pi\in\twobytwo$\,). In the right hand picture of Figure~\ref{fig-squintb}, the points forming a basis element of $\C$ that defines the line $\ell_L$ ensures that in any extremal triple $(v,r,\ell)$, $r$ is lower than $\ell$. Thus $\pi$ consists of (i) a basis element of \tikz[scale=0.4,baseline=0.4cm-0.5ex]{
\tgrid(1,2);
\tup(1,1);
\tclass(1,2){\footnotesize$\C$};}
which defines $v_R$, (ii) a basis element of
\tikz[scale=0.4,baseline=0.4cm-0.5ex]{
\tgrid(1,2);
\tup(1,1);
\tclass(1,2){\footnotesize$\D$};}
to define $v_L$, (iii) a copy of $21$ to define $r_R$, and (iv) a basis element of $\C$ to define $\ell_L$.
\end{proof}

A similar approach, of bounding the number of possible left and right h-lines, can be applied for the other two cases, so we only sketch the proofs.

\begin{lemma}\label{lem-caseb}
For finitely based classes $\C$ and $\D$, the class
\[\E=
\tikz[scale=0.5,baseline=0.5cm-0.5ex]{
\tgrid(2,2);
\tup(1,1);
\tdown(2,1);
\tclass(1,2){$\C$};
\tclass(2,2){$\D$};
}\]
is finitely based.
\end{lemma}

\begin{proof}[(sketch)]
We need only consider relative basis elements of $\E$ that lie in \squintA, as the argument for \squintB\ is symmetric. Thus, consider a basis element $\pi\in\squintA$ of $\E$.

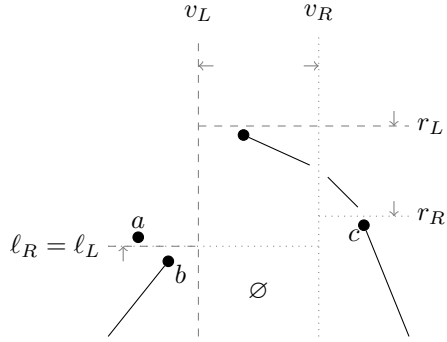
\begin{figure}
\centering
\begin{tikzpicture}[scale=0.4]
\draw [gray, dashed] (3,0) -- (3,10);
\draw [gray, dashed] (0,3) -- (3,3);
\draw [gray, dashed] (3,7) -- (10,7);
\draw [gray, dotted] (7,0) -- (7,10);
\draw [gray, dotted] (0,3) -- (7,3);
\draw [gray, dotted] (7,4) -- (10,4);
\draw [gray, ->] (3.5,9) -- (3,9);
\draw [gray, ->] (6.5,9) -- (7,9);
\draw [gray, ->] (.5,2.5) -- (.5,3);
\draw [gray, ->] (9.5,4.5) -- (9.5,4);
\draw [gray, ->] (9.5,7.5) -- (9.5,7);
\node (a) at (8.5,3.7) [label=-135:$c$] {};
\node (b) at (4.5,6.7) {};
\node (c) at (2,2.5) [label=-45:$b$] {};
\node (d) at (1,3.3) [label=90:$a$] {};
\draw (a) -- (10,0);
\draw (b) -- (6.7,5.7);
\draw (8.3,4.3) -- (7.3,5.3);
\draw (0,0) -- (c);
\node at (10,7) [invis_nd,label=right:$r_L$] {};
\node at (10,4) [invis_nd,label=right:$r_R$] {};
\node at (0,3) [invis_nd,label=left:{$\ell_R=\ell_L$}] {};
\node at (3,10) [invis_nd,label=above:$v_L$] {};
\node at (7,10) [invis_nd,label=above:$v_R$] {};
\node [invis_nd] at (5,1.5) {$\varnothing$};
\end{tikzpicture}
\caption{The left h-line $\ell_R$ is defined by the points $a$ and $b$ which form a copy of 21. Both $a$ and $b$ must lie to the left of $v_L$, so this also defines $\ell_L$.}
\label{fig-caseb}
\end{figure}

Define the division triples $(v_R,r_R,\ell_R)$ and $(v_L,r_L,\ell_L)$ recognising $\pi\in\squintA$ by choosing $v_R$ to be the rightmost v-line, and $v_L$ the leftmost, and then selecting $r_L$ and $r_R$ as low as possible, and $\ell_L$ and $\ell_R$ as high as possible.

We claim that $\ell_R$ and $\ell_L$ have the same height. In Figure~\ref{fig-caseb}, the point $c$ which defines the line $r_R$, forces the region below $\ell_R$ and between $v_L$ and $v_R$ to be empty. Consequently, the pair of points $a$ and $b$ (which forms a copy of 21 and hence defines the height of $\ell_R$) must lie to the left of $v_L$. This means that $a$ and $b$ also define the highest position of \emph{every} left h-line $\ell$ in a division triple $(v,r,\ell)$ recognising $\pi\in\squintA$.

The proof concludes by noting that we can demonstrate $\pi\not\in\E$ by the following points: (i) a basis element of \tikz[scale=0.4,baseline=0.4cm-0.5ex]{
\tgrid(1,2);
\tup(1,1);
\tclass(1,2){\footnotesize$\C$};}
which defines $v_R$, (ii) a basis element of
\tikz[scale=0.4,baseline=0.4cm-0.5ex]{
\tgrid(1,2);
\tdown(1,1);
\tclass(1,2){\footnotesize$\D$};}
to define $v_L$, (iii) a copy of $21$ to define $\ell_R$, and (iv) a basis element of $\D$ to define $r_R$.
\end{proof}

\begin{lemma}\label{lem-casec}
For finitely based classes $\C$ and $\D$, the class
\[\E=
\tikz[scale=0.5,baseline=0.5cm-0.5ex]{
\tgrid(2,2);
\tdown(1,1);
\tup(2,1);
\tclass(1,2){$\C$};
\tclass(2,2){$\D$};
}\]
is finitely based.
\end{lemma}

\begin{proof}[(sketch)]
As before, by symmetry it suffices to consider a relative basis element $\pi\in\squintA$ of $\E$. Define the division triples $(v_R,r_R,\ell_R)$ and $(v_L,r_L,\ell_L)$ as in earlier proofs.

We claim that in any division triple $(v,r,\ell)$ recognising $\pi\in\squintA$ where $\ell$ is as high as possible, $\ell$ has the same height as either $\ell_L$ or $\ell_R$. The situation is illustrated in Figure~\ref{fig-casec}: if $v$ lies to the right of the point $a$ then $\ell$ can be no higher than $\ell_R$. On the other hand, if $v$ lies to the left of $a$, then the only available copy of 12 has $b$ as the `2', so $\ell$ has the same height as $\ell_L$.

\begin{figure}[h]
\centering
\begin{tikzpicture}[scale=0.4]
\draw [gray, dashed] (3,0) -- (3,10);
\draw [gray, dashed] (0,6) -- (3,6);
\draw [gray, dashed] (3,8) -- (10,8);
\draw [gray, dotted] (7,0) -- (7,10);
\draw [gray, dotted] (0,3) -- (7,3);
\draw [gray, dotted] (7,5) -- (10,5);
\draw [gray, ->] (3.5,9) -- (3,9);
\draw [gray, ->] (6.5,9) -- (7,9);
\draw [gray, ->] (2.5,2.5) -- (2.5,3);
\draw [gray, ->] (2.5,5.5) -- (2.5,6);
\draw [gray, ->] (7.5,5.5) -- (7.5,5);
\draw [gray, ->] (7.5,8.5) -- (7.5,8);
\node (a) at (8,4.7) [label=-45:$c$] {};
\node (b) at (9.5,7.7) {};
\node (c) at (1.5,2.5) {};
\node (d) at (5.7,3.3) [label=135:$a$] {};
\node (e) at (5,1) {};
\node at (2.5,6.3) [label=135:$b$] {};
\draw (a) -- (7.3,4.3);
\draw (b) -- (8.5,5.3);
\draw (2.5,1.5) -- (c);
\draw (6.7,4) -- (d);
\draw (1.3,3.3) -- (0.5,5.7);
\node at (10,8) [invis_nd,label=right:$r_L$] {};
\node at (10,5) [invis_nd,label=right:$r_R$] {};
\node at (0,6) [invis_nd,label=left:{$\ell_L$}] {};
\node at (0,3) [invis_nd,label=left:{$\ell_R$}] {};
\node at (3,10) [invis_nd,label=above:$v_L$] {};
\node at (7,10) [invis_nd,label=above:$v_R$] {};
\end{tikzpicture}
\caption{The left h-line $\ell_R$ is defined by the points $a$ and $b$ which form a copy of 12. Since $a$ lies to the left of $v_L$, the left h-line $\ell_L$ can be no higher than $\ell_R$.}
\label{fig-casec}
\end{figure}

With these two left h-lines defined, we need only identify two copies of basis elements of $\D$ to define corresponding lowest right h-lines in each case. Thus, $\pi\not\in\E$ is identified by the following points: (i) a basis element of \tikz[scale=0.4,baseline=0.4cm-0.5ex]{
\tgrid(1,2);
\tdown(1,1);
\tclass(1,2){\footnotesize$\C$};}
to defines $v_R$, (ii) a basis element of
\tikz[scale=0.4,baseline=0.4cm-0.5ex]{
\tgrid(1,2);
\tup(1,1);
\tclass(1,2){\footnotesize$\D$};}
to define $v_L$, (iii) at most two copies of $21$ to define $\ell_R$ and $\ell_L$, and (iv) at most two basis elements of $\D$ to define $r_R$ and $r_L$.
\end{proof}

\begin{proof}[of Theorem~\ref{thm-twobytwo-mono}]
First, the only $2\times 2$ monotone grid classes whose row-column graphs are not forests (and hence finitely based by~\cite{albert:geometric-grid-:}) are those where all four cells are non-empty.

Any such $2\times 2$ monotone grid class can be described as a grid class in one of the three forms covered by Lemmas~\ref{lem-casea},~\ref{lem-caseb} and~\ref{lem-casec}, upon taking the classes $\C$ and $\D$ to be $\Av(12)$ or $\Av(21)$, and possibly appealing to symmetry.
\end{proof}

\section{Concluding remarks}

\paragraph{Non-monotone 2 $\times$ 2 grids} One obvious question arising from this work is how far one might be able to extend Theorem~\ref{thm-twobytwo} within the context of $2\times 2$ grids: in particular, can one replace the two monotone classes in the lower row by something more general? Any approach to this question would need to bear in mind that there do exist $2\times 2$ grid classes which are \emph{not} finitely based, even though each entry of the matrix is finitely based. The primary example of this, given both in Murphy's PhD thesis~\citep{murphy:restricted-perm:} and in~\citep{atkinson:restricted-perm:a}, is 
\[
\tikz[scale=0.5,baseline=0.5cm-0.5ex]{
\tgrid(2,2);
\tclass(1,1){$\C$};
\tclass(2,2){$\C$};
\tclass(1,2){$\varnothing$};
\tclass(2,1){$\varnothing$};
}\]
where $\C=\Av(321654)$. (Note this example is more normally written as a \emph{direct sum}, $\C\oplus\C$.) This example can likely be adapted to produce other instances where the grid class is not finitely based, even though its individual entries are.

\paragraph{Larger grids} There are a number of difficulties encountered when one tries to extend our results here to larger grids. Even in the ``next'' case of $2\times 3$ grids, there seems to be no obvious analogue to Lemma~\ref{lem-squint-intersection} to enable us to consider relative bases inside some larger class. The primary issue is that our proof relied on the fact that the heights of all possible left-h-lines (or, analogously, right-h-lines) form a contiguous set of values, but this need no longer be the case. 

\paragraph{Acknowledgements} We are grateful to Mike Atkinson for several fruitful discussions about this problem, from which most of the ideas for this note emerged.
\bibliographystyle{abbrvnat}
\bibliography{../../refs}

\end{document}